\documentclass[11pt]{amsart}
\usepackage{amsmath,amssymb,amsthm,mathtools}
\usepackage[margin=1in]{geometry}
\usepackage{mathrsfs}
\usepackage{graphicx}
\usepackage{hyperref}
\usepackage{tikz}
\usetikzlibrary{shapes.geometric, arrows.meta, positioning}
\usepackage{orcidlink} 

\newtheorem{theorem}{Theorem}[section]
\newtheorem{lemma}[theorem]{Lemma}

\newtheorem{corollary}[theorem]{Corollary}
\newtheorem{proposition}[theorem]{Proposition}

\newtheorem{definition}[theorem]{Definition}
\newtheorem{remark}[theorem]{Remark}
\newtheorem{counterexample}[theorem]{Counter-example}

\def\union{\mathop{\cup }}
\def\inter{\mathop{\cap }}
\def\petitoplus{\mathop{\bigoplus}}

\begin{document}
\title[Pego theorem for Hilbert space-valued functions on compact groups]{Pego theorem for Hilbert space-valued functions on compact groups}
\author{Anat\'e Kodjovi Lakmon$^1$ and Yaogan Mensah$^{1,2}$}

\address{$^1$ Department of Mathematics, University of Lom\'e, Lom\'e,  Togo}
\address{$^2$ ICMPA-Unesco-Chair, University of Abomey-Calavi, Calavi, Benin}
\email{davidlakmon@gmail.com, klakmon@univ-lome.tg}
\email{mensahyaogan2@gmail.com, ymensah@univ-lome.tg }

\dedicatory{In memory of Professor Koffi Kenny Siggini}

\begin{abstract}
We prove a Hilbert space-valued analogue of Pego's compactness theorem on compact groups. For square-integrable functions taking values in a Hilbert space  rather than in the complex numbers, we show that a bounded family is precompact exactly when it is simultaneously well behaved in two complementary senses: its members do not change much under small translations of the group, and their Fourier coefficients decay uniformly across the family. This  equivalence holds without restriction when  the Hilbert space is finite-dimensional, and it specializes to the known scalar-valued theorem when the Hilbert space is just the complex numbers. We then construct an explicit example showing that the equivalence genuinely breaks down once the Hilbert space is allowed to be infinite-dimensional.  To repair this, we introduce a uniform tightness condition and we show that under this extra hypothesis the equivalence is restored regardless of the dimension of the Hilbert space. Along the way we establish the Plancherel isometry, the Hausdorff-Young inequality and its inverse for this vector-valued Fourier transform. 
\end{abstract}


\maketitle

\section{Introduction}

Characterizing the precompact subsets of a function space is one of the oldest problems in analysis, and it has a distinctive shape: precompactness tends to be equivalent to controlling a family simultaneously in two dual variables. The Arzel\`a-Ascoli theorem is the prototype for continuous functions on a compact space: a family is precompact once it is bounded and equicontinuous. The Riesz-Kolmogorov theorem (see \cite{Hanche-Olsen} for historical review) plays the same role for $L^p(\mathbb{R}^n)$, and Weil~\cite{Weil} later extended it to Lebesgue spaces over locally compact groups. Later, G\'orka and P\'ospiech extended  the Riesz-Kolmogorov theorem to Banach function spaces on locally compact groups \cite{Gorka2019}.  Underneath these results lies a recurring duality between a time domain and a frequency domain condition: control in one variable, plus control in its Fourier dual, together force compactness.

In 1985, Pego~\cite{Pego} made this duality fully explicit. He showed that a bounded subset of $L^2(\mathbb{R}^n)$ is precompact if and only if and only if the functions are uniformly continuous under translation and their transforms are uniformly continuous under translation in frequency. Beyond its intrinsic interest, this criterion has an appealing consequence for information theory: the rate of decay controls how many degrees of freedom are needed to approximate every function in the family to a given accuracy.

Pego's theorem has since been transplanted into several other transform settings: the short-time Fourier and wavelet transforms \cite{Dorfler2002}, the Laplace transform  \cite{Krukowski2020}, the Laguerre and Hankel transforms \cite{Horvath2022} and, closer to the present work, onto topological groups. G\'orka~\cite{Gorka2014, Gorka2016} extended it to locally compact abelian groups. Kumar~\cite{Kumar2024} extended it further  to compact groups that need not be abelian, replacing the ordinary Fourier transform with the operator-valued transform attached to the group's unitary dual. In a complementary direction, Krukowski~\cite{Krukowski} proved an $L^1$, rather than $L^2$, analogue over locally compact abelian groups, by identifying $L^1(G)$ with a subalgebra of continuous functions vanishing at infinity on the dual and adapting the Arzel\`a-Ascoli theorem itself to that setting.

Every one of these results, however, concerns scalar-valued functions. This paper develops the natural next step: functions on a compact group $G$ that take values not in $\mathbb{C}$ but in a fixed complex Hilbert space $\mathcal{H}$. This is more than a routine generalization. The key technical fact that makes Kumar's argument work, namely, that the Fourier transform is an isometry, turns out to depend essentially on $\mathcal{H}$ being a genuine Hilbert space. We give an explicit example (Counter-example~\ref{ex:planch-fails-Cstar}) showing that this isometry fails as soon as $\mathcal{H}$ is replaced by a more general Banach space. Working with an actual Hilbert space restores the isometry (Theorem~\ref{theo:plancherel}), and from it we recover the full Hausdorff-Young inequality and its inverse (Theorem~\ref{prop:HY-final} and Theorem~\ref{theo:invHY}), obtained by complex interpolation \`a la Calder\'on~\cite{BerghLofstrom} between the two endpoint cases $p=1$ and $p=2$.

With this foundation in place, we characterize the precompact subsets of the vector-valued space $L^2(G,\mathcal{H})$ in terms of uniform equicontinuity under translation and uniform decay of the vector-valued Fourier coefficients (Theorem~\ref{thm:main}). This equivalence holds outright when $\mathcal{H}$ is finite-dimensional, and it reduces exactly to the Pego theorem proved by Kumar in \cite{Kumar2024} when $\dim \mathcal{H} = 1$. We then exhibit an explicit counterexample (Counter-example~\ref{ex:counterexample}) showing that the equivalence genuinely fails once $\mathcal{H}$ is infinite-dimensional: a family can satisfy both the translation and the Fourier-decay conditions while failing to be precompact, simply because bounded sets in infinite-dimensional Hilbert spaces need not be precompact. To repair this, we introduce a uniform tightness condition which is the  vector-valued counterpart of Krukowski's equivanishing condition~\cite{Krukowski}:  supply exactly the missing control in the "value" direction that translation-equicontinuity and frequency-decay do not provide on their own once the space of values stops being finite-dimensional. Under this extra hypothesis, we show that the full  equivalence holds regardless of the dimension of $\mathcal{H}$ (Theorem~\ref{thm:main-tight}).

The paper is organized as follows. Section~\ref{sec:prelim} recalls the needed background on representations of compact groups and scalar-valued Fourier analysis,  including  the Pego  theorem provided by Kumar (Theorem~\ref{thm:kumar}). Section~\ref{sec: main results} contains all of the new results. It first develops the Fourier transform of Hilbert space-valued functions (following the approach of \cite{Assiamoua1989}), establishes the Plancherel isometry and the Hausdorff-Young inequality and its inverse, and exhibits the parallelogram-law obstruction that prevents these results from extending beyond Hilbert space-valued functions. It then turns to compactness: the equicontinuity/decay equivalence, the finite-dimensional Pego theorem, an explicit counterexample showing its failure in infinite dimensions, and the tightness-restored version that holds in full generality.
  
\section{Preliminary notes}\label{sec:prelim}
\subsection{Representations of compact groups and Fourier analysis}

This section delivers the core insights needed to understand the article. We are chiefly concerned with Fourier analysis of scalar-valued functions on a compact group.
 For more details, we refer \cite{Folland} and the preliminary notes of \cite{Kumar2024}
 
\noindent Let $G$ be a compact Hausdorff group with identity $e$. There is a unique
regular Borel probability measure $m_G$  on $G$, invariant under left and right
translation and under inversion.  
The measure $m_G$ is called the Haar measure of $G$. Throughout this  paper, 
 $L^p(G), \,1\le p\le\infty$, denote the usual Lebesgue spaces  with respect to the measure $m_G$ and endowed with their natural $L^p$-norm.

\noindent A unitary representation of $G$ is a pair $(\sigma,H_\sigma)$, $H_\sigma$ a
complex Hilbert space, $\sigma:G\to U(H_\sigma)$ a homomorphism from $G$ into the unitary group
of $H_\sigma$, continuous in the sense that $x\mapsto\sigma(x)\xi$ is continuous for
every $\xi\in H_\sigma$. Its dimension is $d_\sigma:=\dim H_\sigma$. A subspace $W$ of $H_\sigma$ is said to be invariant by $\sigma$ if $$\forall x\in G,\,\forall \xi\in W, \sigma(x)\xi\in W.$$
The representation $(\sigma,H_\sigma)$ issaid to be  irreducible if $\{0\}$ and $H_\sigma$ are its only closed
invariant subspaces. Two representations $(\sigma_i,H_{\sigma_i})$,
$i=1,2$, are called unitarily equivalent if there is a unitary isomorphism  $T:H_{\sigma_1}\to
H_{\sigma_2}$ such that $T\sigma_1(x)=\sigma_2(x)T$ for all $x\in G$.
The following theorem is well-known. 
\begin{theorem}(\cite[page 126]{Folland})
If $G$ is compact, every irreducible unitary representation of $G$ is
finite-dimensional.
\end{theorem}

\noindent Let $\widehat{G}$ denote the set of unitary equivalence classes of irreducible unitary
representations of $G$.  It is a discrete set.
Fix, for each $\sigma\in\widehat{G}$, an orthonormal basis $\{\xi_1^\sigma,\dots,
\xi_{d_\sigma}^\sigma\}$ of $H_\sigma$. The matrix coefficients of $\sigma$ are the functions $u^\sigma_{ij}$ defined by
\[
u^\sigma_{ij}(x)=\langle\sigma(x)\xi_i^\sigma,\xi_j^\sigma\rangle_{H_\sigma},\quad
x\in G,\ 1\le i,j\le d_\sigma.
\]
Each $u^\sigma_{ij}\in C(G)$, with $|u^\sigma_{ij}(x)|\le1$ for all $x$.
We recall the following two theorems. 
\begin{theorem}[Schur orthogonality relations](\cite[page 129]{Folland})
For $\sigma,\tau\in\widehat{G}$, $1\le i,j\le d_\sigma$, $1\le k,l\le d_\tau$,
\[
\int_G u^\sigma_{ij}(x)\,\overline{u^\tau_{kl}(x)}\,dm_G(x)=
\frac{\delta_{\sigma\tau}\,\delta_{ik}\,\delta_{jl}}{d_\sigma}.
\]

\end{theorem}

\begin{theorem}[Peter-Weyl](\cite[page 133]{Folland})
The family $\{\sqrt{d_\sigma}\,u^\sigma_{ij} : \sigma\in\widehat{G},\ 1\le i,j\le d_\sigma\}$
is an orthonormal basis of $L^2(G)$.
\end{theorem}

For $y\in G$ and $f:G\to\mathbb C$, the right translation  is defined by  $$R_yf(x):=f(xy).$$
The following proposition holds. 
\begin{proposition}\label{proposition:translation-isometry}
For every $y\in G$ and $1\le p<\infty$, $R_y$ is an isometry of $L^p(G)$.
\end{proposition}

For $\sigma\in\widehat{G}$, let $B(H_\sigma)$ denote the spaces of bounded linear operators on $B(H_\sigma)$ with the operator norm. For $1\le p<\infty$,  

$$\ell^p\text{-}\bigoplus_{\sigma\in\widehat{G}}B(H_\sigma)=\Big\{(T_\sigma):
\sum_{\sigma\in\widehat{G}}d_\sigma\|T_\sigma\|_{B(H_\sigma)}^p<\infty\Big\}$$
with the norm 
$$\|(T_\sigma)\|_{\ell^p\text{-}\bigoplus B(H_\sigma)}=\Big(\sum_{\sigma\in\widehat{G}}d_\sigma\|T_\sigma\|_{B(H_\sigma)}^p\Big)^{1/p}.$$

$$\ell^\infty\text{-}\bigoplus_{\sigma\in\widehat{G}}B(H_\sigma):=\Big\{(T_\sigma):
\sup_{\sigma\in\widehat{G}}\|T_\sigma\|_{B(H_\sigma)}<\infty\Big\}$$
with the norm
$$\|(T_\sigma)\|_{\ell^\infty\text{-}\bigoplus B(H_\sigma)}=\sup_\sigma\|T_\sigma\|_{B(H_\sigma)}.$$
And, 
 $c_0\text{-}\petitoplus\limits_{\sigma\in\widehat{G}}B(H_\sigma)$ is the closed subspace of
$\ell^\infty\text{-}\petitoplus\limits_{\sigma\in\widehat{G}} B(H_\sigma)$  consisting of   $(T_\sigma)$ such that 
$T_\sigma$ tends to $0$ as $\sigma$ goes to $\infty$  i.e. for every $\varepsilon>0$ the set
$\{\sigma\in \widehat{G}:\|T_\sigma\|_{B(H_\sigma)}>\varepsilon\}$ is finite.

For $f\in L^1(G)$, the Fourier transform of $f$ is given by
$$
\widehat f(\sigma)=\int_G f(x)\,\sigma(x)^*dm_G(x)\in B(H_\sigma),\quad\sigma\in\widehat{G}.
$$

\begin{theorem}
The Fourier transformation  $\mathscr{F}: f\mapsto\widehat f$ is injective and bounded from $L^1(G)$ into
$\ell^\infty\text{-}\petitoplus\limits_{\sigma \in \widehat{G}} B(H_\sigma)$. Moreover,    
$\widehat f\in c_0\text{-}\petitoplus\limits_{\sigma\in\widehat{G}}B(H_\sigma)$  (the Riemann-Lebesgue lemma).
\end{theorem}
\begin{theorem}[Fourier inversion]
For $f\in L^2(G)$,
\[
f(x)=\sum_{\sigma\in\widehat{G}}d_\sigma\,\operatorname{tr}\big(\widehat f(\sigma)\,\sigma(x)\big),
\quad x\in G,
\]
where convergence holds in the $L^2(G)$-norm.
\end{theorem}

\begin{theorem}\label{scalar:plancherel}
The Fourier transform is an isometric isomorphism from $L^2(G)$ onto
$\ell^2\text{-}\petitoplus\limits_{\sigma\in\widehat{G}}B_2(H_\sigma)$:
$$
\|f\|_2^2=\sum_{\sigma\in\widehat{G}}d_\sigma\,\|\hat f(\sigma)\|_{B_2(H_\sigma)}^2,$$
where $B_2(H_\sigma)$ is the Hilbert-Schmidt class on $H_\sigma$.
\end{theorem}

For $f,g\in L^1(G)$, the convolution of $f$ by $g$ is defined by 
$$
(f*g)(x):=\int_G f(xy^{-1})\,g(y)\,dm_G(y).$$
The following result (convolution theorem) holds:
$$\widehat{f*g}(\sigma)=\widehat{G}(\sigma)\,
\widehat f(\sigma),\quad\sigma\in\widehat{G}.
$$

Let us also mention that for $y\in G$ and $f\in L^p(G)$, 
$$
\widehat{R_yf}(\sigma)=\sigma(y)\,\widehat f(\sigma),\quad\sigma\in\widehat{G}.
$$

\subsection{Pego theorem on compact groups: the scalar case}

In this subsection, we present Pego's theorem as formulated by Kumar in \cite{Kumar2024}. Prior to that, we introduce the definitions essential for its understanding.

Let $(X, d)$ be a metric space.
\begin{definition}
A subset $A$ of $X$ is said to be  {totally bounded} if, for every $\varepsilon > 0$, there exists a finite set of points $\{x_1, x_2, \dots, x_n\} \subset X$ such that:
$$
A \subset \union\limits_{i=1}^{n} B(x_i, \varepsilon),
$$
where $B(a,r)$ denotes the open ball centered at $a$ with radius $r$.
\end{definition}
\begin{definition}
A subset $A$ of $X$ is said to be {precompact} if its topological closure, denoted as $\overline{A}$, is compact.
\end{definition}

\noindent The Hausdorff criterion for precompactness states that a subset of a metric space is precompact if and only if it is totally bounded. 

\begin{definition}
The set $K\subset L^p(G)$ is said to be uniformly $L^p(G)$-equicontinuous if for every
$\varepsilon>0$ there is an open neighborhood $O$ of $e$ such that $\|R_yf-f\|_{L^p(G)}<\varepsilon$
for $f\in K,\ y\in O$.
\end{definition}

Let $B_p(H_\sigma)$ denotes   the Schatten $p$-class of bounded linear operators on $H_\sigma$ endowed with norm  $\|T\|_{B_p(H_\sigma)}=\left(\text{tr}(|T|^p)\right)^{\frac{1}{p}}$.

\begin{definition}
The set $F\subset\ell^p\text{-}\petitoplus\limits_{\sigma\in \widehat{G}} B_p(H_\sigma)$ is said to have uniform
$\ell^p\text{-}\petitoplus\limits_{\sigma\in \widehat{G}} B_p(H_\sigma)$-decay if for every $\varepsilon>0$ there
is a finite set $S\subset\widehat{G}$ such that  $\|\varphi\|_{\ell^p\text{-}\petitoplus_{\sigma \in \widehat{G}\setminus S}B_p(H_\sigma)}<\varepsilon$ for all $\varphi\in F$. 

\end{definition}

\begin{theorem} \cite[Theorem 1.1]{Kumar2024}\label{thm:kumar}
Let $K$ be a bounded subset of $L^2(G)$. The following assertions are equivalent:
\begin{enumerate}
\item[(i)] $K$
is precompact;
\item[(ii)] $K$ is uniformly $L^2(G)$-equicontinuous;
\item[(iii)] $\widehat K$ has uniform $\ell^2\text{-}\petitoplus\limits_{\sigma \in \widehat{G}} B_2(H_\sigma)$-decay.
\end{enumerate}
\end{theorem}
\section{Pego theorem for Hilbert-space-valued functions}\label{sec: main results}
We now establish several results, leading up to the generalization of  Theorem~\ref{thm:kumar} to Hilbert space-valued functions.

 Let $G$ be a compact group. Fix a  complex Hilbert space $\mathcal{H}$. The Bochner spaces of $p$-strongly integrable $\mathcal{H}$-valued functions on $G$ are denoted by $L^p(G,\mathcal{H}), \, 1\le p<\infty$. The norm on $L^p(G,\mathcal{H})$ is given by 
 
$$ \|f\|_{L^p(G,\mathcal{H})}=\left(\int_G\|f(x)\|_{\mathcal{H}}^pdm_G(x)\right)^{\frac{1}{p}}.$$
In particular, $L^2(G,\mathcal{H})$ is a complex  Hilbert space with the inner product
$$\langle f,g\rangle_{L^2(G,\mathcal{H})}=\int_G\langle f(x),g(x)\rangle_{\mathcal{H}}dm_G(x).$$
\begin{definition}\cite{Assiamoua1989}
For $f\in L^1(G,\mathcal{H})$, the Fourier transform  $\widehat f=(\widehat f(\sigma))_{\sigma\in\widehat{G}}$
is the family of sesquilinear maps $\widehat f(\sigma):H_\sigma\times H_\sigma\to\mathcal{H}$ given by 
\[
\widehat f(\sigma)(\xi,\eta)=\int_G\langle\sigma(x)^*\xi,\eta\rangle_{H_\sigma}\,f(x)\,dm_G(x),
\qquad\xi,\eta\in H_\sigma,
\]
where $\langle \cdot, \cdot\rangle_{H_\sigma}$ is the scalar product in $H_\sigma$.
\end{definition}

For $f\in L^2(G,\mathcal{H})$, the inversion formula holds:
\begin{equation}\label{eq:inversionformula}
f(x)=\sum_{\sigma\in\widehat{G}}d_\sigma\sum_{i=1}^{d_\sigma}\sum_{j=1}^{d_\sigma}u^\sigma_{ij}(x)\widehat f(\sigma)(\xi_j^\sigma,
\xi_i^\sigma),\quad x\in G.
\end{equation}

In the sequel, we will write
$$\varphi(\sigma)_{ij}=\varphi(\sigma)(\xi_j^\sigma,
\xi_i^\sigma).$$ In particular, 
$$\widehat f(\sigma)_{ij}=\widehat f(\sigma)(\xi_j^\sigma,
\xi_i^\sigma).$$

We consider the Assiamoua spaces $\mathscr{S}_p(\widehat{G}, \mathcal{H}),\, 1\le p\le \infty$, defined as follows \cite{Mensah2024}:
\begin{itemize}
\item for $1\le p<\infty$, 
$$\mathscr{S}_p(\widehat{G},\mathcal{H})=\Big\{\varphi=(\varphi(\sigma))_{\sigma\in \widehat{G}}:
\sum_{\sigma\in \widehat{G}} d_\sigma\sum_{i=1}^{d_\sigma}\sum_{j=1}^{d_\sigma}\|\varphi(\sigma)_{ij}\|_\mathcal{H}^p<\infty\Big\},
$$
with the norm 
$$ \|\varphi\|_{\mathscr{S}_p(\widehat{G},\mathcal{H})}=\left(\sum_{\sigma\in \widehat{G}} d_\sigma\sum_{i=1}^{d_\sigma}\sum_{j=1}^{d_\sigma}\|\varphi(\sigma)_{ij}\|_\mathcal{H}^p\right)^{\frac{1}{p}}.
$$
\item $$\mathscr{S}_\infty(\widehat{G},\mathcal{H})=\Big\{  \varphi=(\varphi(\sigma))_{\sigma\in \widehat{G}}: \sup_{\sigma \in \widehat{G}}\|\varphi(\sigma)\| <\infty \Big\},$$

with the norm 

$$\|\varphi\|_{\mathscr{S}_\infty(\widehat{G},\mathcal{H})}=\sup_{\sigma \in \widehat{G}}\|\varphi(\sigma)\|.$$
Here,  
$$\|\varphi(\sigma)\|=\sup\Big\{\|\varphi(\sigma)(\xi,\eta)\|_{\mathcal{H}}: \|\xi\|_{H_\sigma}\leq 1, \|\eta\|_{H_\sigma}\leq 1\Big\}.$$
\end{itemize}

\noindent Each $\mathscr{S}_p(\widehat{G},\mathcal{H}), \, 1\leq p\leq \infty$ is a complex Banach
space, and in particular $\mathscr{S}_2(\widehat{G},\mathcal{H})$ is a Hilbert space with the inner product
$$\langle\varphi,\psi\rangle_{\mathscr{S}_2}=\sum_{\sigma\in \widehat{G}} d_\sigma\sum_{i=1}^{d_\sigma}\sum_{j=1}^{d_\sigma}\langle
\varphi(\sigma)_{ij},\psi(\sigma)_{ij}\rangle_\mathcal{H}.$$

\noindent We provide an alternative proof for \cite[Theorem 4.8]{Assiamoua1989} concerning the analogue of the Plancherel theorem for this Fourier transform. The counter-example following  Theorem~\ref{theo:plancherel} illustrates why we restrict our attention to a Hilbert space-valued functions. 
\begin{theorem}\label{theo:plancherel}
Let $G$ be a compact group and let $\mathcal{H}$ be a complex Hilbert space.   The Fourier transformation $\mathscr{F}:L^2(G,\mathcal{H})\to \mathscr{S}_2(\widehat{G},\mathcal{H})$ is an isometric isomorphism.
\end{theorem}

\begin{proof}
Fix an orthonormal basis $\{h_k\}_{k\in \mathfrak{K}}$ of the Hilbert space $\mathcal{H}$. By the Peter-Weyl theorem,  the family $\{\sqrt{d_\sigma}\,u^\sigma_{ij} :
\sigma\in\widehat{G},\ 1\le i,j\le d_\sigma\}$ is an orthonormal basis of $L^2(G)$.  Since
$L^2(G,\mathcal{H})$ is isomorphic to $L^2(G)\otimes\mathcal{H}$ canonically as Hilbert spaces, the family
\[
\big\{\sqrt{d_\sigma}\,u^\sigma_{ij}\otimes h_k \;:\; \sigma\in\widehat{G},\ 1\le i,j\le
d_\sigma,\ k\in \mathfrak{K}\big\}
\]
is an orthonormal basis of $L^2(G,\mathcal{H})$. For $f\in  L^2(G,\mathcal{H})$,  the basis coefficient $c^\sigma_{ij,k}$ of $f$ against the vector $\sqrt{d_\sigma}\,u^\sigma_{ij}\otimes h_k$ is
\begin{equation}\label{Coefficient}
c^\sigma_{ij,k}=\big\langle f,\ \sqrt{d_\sigma}\,u^\sigma_{ij}\otimes h_k
\big\rangle_{L^2(G,\mathcal{H})}
=\sqrt{d_\sigma}\int_G \overline{u^\sigma_{ij}(x)}\,\big\langle f(x),h_k\big\rangle_\mathcal{H}
\,dm_G(x).
\end{equation}

Now, we aim to  relate $c^\sigma_{ij,k}$ to the Fourier coefficient $\widehat{f}(\sigma)_{ij}$.
We have
\[
\big\langle\sigma(x)^*\xi_j^\sigma,\xi_i^\sigma\big\rangle_{H_\sigma}
=\big\langle\xi_j^\sigma,\sigma(x)\xi_i^\sigma\big\rangle_{H_\sigma}
=\overline{\big\langle\sigma(x)\xi_i^\sigma,\xi_j^\sigma\big\rangle}_{H_\sigma}
=\overline{u^\sigma_{ij}(x)}.
\]

Hence, directly from the definition of the Fourier transform, we can write 
\[
\widehat f(\sigma)_{ij}
=\int_G\overline{u^\sigma_{ij}(x)}\,f(x)\,dm_G(x).
\]
Taking the inner product with $h_k$ and compairing with (\ref{Coefficient}), we obtain 
\[
\big\langle\widehat f(\sigma)_{ij},h_k\big\rangle_\mathcal{H}
=\int_G\overline{u^\sigma_{ij}(x)}\,\big\langle f(x),h_k\big\rangle_\mathcal{H}\,dm_G(x)
=\frac{c^\sigma_{ij,k}}{\sqrt{d_\sigma}},
\]
i.e.
\begin{equation}\label{CoefficientC}
c^\sigma_{ij,k}=\sqrt{d_\sigma}\;\big\langle \widehat f(\sigma)_{ij},h_k\big\rangle_\mathcal{H}.
\end{equation}
So the basis coefficients of $f$ are, up to the fixed scalar factor $\sqrt{d_\sigma}$,
exactly the coordinates of the Fourier coefficients $\widehat f(\sigma)_{ij}$ against the basis $\{h_k\}$ of $\mathcal{H}$.

\noindent Applying Parseval's identity for
the orthonormal basis
 $\{\sqrt{d_\sigma}u^\sigma_{ij}\otimes h_k:\sigma\in \widehat{G}, 1\le i,j\le d_\sigma, k\in \mathfrak{K}\}$ of $L^2(G,\mathcal{H})$
by using (\ref{CoefficientC}), we attain
\[
\|f\|_{L^2(G,\mathcal{H})}^2=\sum_{\sigma\in \widehat{G}}\sum_{i=1}^{d_\sigma}\sum_{j=1}^{d_\sigma}\sum_{k\in \mathfrak{K}}|c^\sigma_{ij,k}|^2
=\sum_{\sigma\in \widehat{G}}d_\sigma\sum_{i=1}^{d_\sigma}\sum_{j=1}^{d_\sigma}\sum_{k\in \mathfrak{K}}\big|\langle\widehat f(\sigma)_{ij},h_k\rangle_\mathcal{H}\big|^2.
\]
For each fixed $\sigma,i,j$, Parseval's identity in $\mathcal{H}$ itself, applied to the vector
$\widehat f(\sigma)_{ij}\in\mathcal{H}$ against the orthonormal basis $\{h_k:k\in \mathfrak{K}\}$, gives
\[
\sum_{k\in \mathfrak{K}}\big|\langle \widehat f(\sigma)_{ij},h_k\rangle_\mathcal{H}\big|^2=\|\widehat f(\sigma)_{ij}\|_\mathcal{H}^2.
\]
Substituting back,
\[
\|f\|_{L^2(G,\mathcal{H})}^2=\sum_{\sigma\in \widehat{G}}d_\sigma\sum_{i=1}^{d_\sigma}\sum_{j=1}^{d_\sigma}\|\widehat f(\sigma)_{ij}\|_\mathcal{H}^2
=\|\widehat f\|_{\mathscr{S}_2(\widehat{G},\mathcal{H})}^2.
\]

Therefore, the Fourier transformation $\mathscr{F}$ is an isometry, hence  an  injection. 
\smallskip

Furthermore, given $\varphi\in \mathscr{S}_2(\widehat{G},\mathcal{H})$, that is,  $$\sum_{\sigma\in \widehat{G}}d_\sigma\sum_{i=1}^{d_\sigma}\sum_{j=1}^{d_\sigma}\|\varphi(\sigma)_{ij}\|_\mathcal{H}^2<\infty, $$
 define scalars $q^\sigma_{ij,k}$ by
$q^\sigma_{ij,k}=\sqrt{d_\sigma}\langle\varphi(\sigma)_{ij},h_k\rangle_\mathcal{H}$. By the same computation as above, $$\sum_{\sigma\in \widehat{G}}\sum_{i=1}^{d_\sigma}\sum_{j=1}^{d_\sigma}\sum_{k\in \mathfrak{K}}
|q^\sigma_{ij,k}|^2=\|\varphi\|_{\mathscr{S}_2(\widehat{G},\mathcal{H})}^2<\infty.$$
Since $L^2(G,\mathcal{H})$ is complete,  the series
\[
\sum_{\sigma\in \widehat{G}}\sum_{i=1}^{d_\sigma}\sum_{j=1}^{d_\sigma}\sum_{k\in \mathfrak{K}}q^\sigma_{ij,k}\;\sqrt{d_\sigma}\,u^\sigma_{ij}\otimes h_k
\]
converges in $L^2(G,\mathcal{H})$ to a well-defined element $g$, and  
 $\widehat{G}(\sigma)_{ij}=\varphi(\sigma)_{ij}$ for every $\sigma,i,j$,
i.e.\ $\widehat{G}=\varphi$. Hence, $f\mapsto\widehat f$ is surjective. 
\end{proof}
In Theorem~\ref{theo:plancherel}, the assumption that \(\mathcal{H}\) is a Hilbert space is essential, without which the Fourier transform may fail to be an isometry, as illustrated by the following counter-example.

\begin{counterexample}\label{ex:planch-fails-Cstar}{\rm
Take $G=\mathbb T$ be the 1-dimensional torus (here $\widehat{G}=\mathbb Z$ and the degrees of the characters are equal to 1)  and $\mathcal{H}=\mathbb C^2$ equipped with
the norm $\|(z_1,z_2)\|_{\mathbb C^2}=\max(|z_1|,|z_2|)$. This norm on $\mathbb C^2$ does not satisfy the parallelogram law, so $\left(\mathbb C^2,\|\cdot\|_{\mathbb C^2} \right)$ is not a Hilbert space. 

Let $u_1(x)=e^{ix}$, $u_2(x)=e^{2ix}$ be orthonormal characters of
$\mathbb T$, and set
$$
h(x)=a_1u_1(x)+a_2u_2(x)$$
where $a_1=(1,0)$ and $a_2=(0,1)$.
Then $h(x)=(e^{ix},e^{2ix})$, so $\|h(x)\|_{\mathbb C^2}=\max(1,1)=1$ for every $x\in\mathbb T$,
giving
\[
\|h\|_{L^2(\mathbb T,\mathbb C^2)}^2=1.
\]

\noindent But  the Fourier coefficients computed directly are $\widehat h(u_1)=a_1$, $\widehat h(u_2)=a_2$,
and all other coefficients vanish. So, 
\[
\|\widehat h\|_{\mathscr{S}_2(\mathbb Z,\mathbb C^2)}^2=\|a_1\|_{\mathbb C^2}^2+\|a_2\|_{\mathbb C^2}^2=1+1=2\neq 1
=\|h\|_{L^2(\mathbb T,\mathbb C^2)}^2.
\]
 Therefore, the  Fourier transform is not
an isometry from $L^2(\mathbb T,\mathbb C^2)$ into $\mathscr{S}_2(\mathbb Z,\mathbb C^2)$. 
}\end{counterexample}
Counter-example~\ref{ex:planch-fails-Cstar}  provides a negative answer to Question 2 in \cite{Mensah2024}. 

\smallskip

The following result is the analogue of the Hausdorff-Young inequality. We prove it via an interpolation method, bypassing the theory of Fourier type \cite{GarciaRubio, Peetre}.

\begin{theorem}\label{prop:HY-final}
Let $G$  be a compact group and  $\mathcal{H}$ a complex Hilbert space. Let   $1\le p\le 2$. If $f\in L^p(G,\mathcal{H})$, then
\[
\|\widehat f\|_{\mathscr{S}_{p'}(\widehat{G},\mathcal{H})}\le\|f\|_{L^p(G,\mathcal{H})},\quad\text{where }
\frac1p+\frac1{p'}=1.
\]
\end{theorem}

\begin{proof}
If $f\in L^1(G,\mathcal{H})$, then $\widehat{f}\in \mathscr{S}_\infty(\widehat{G},\mathcal{H})$. Indeed,
for every $\sigma\in\widehat{G}$ and $\xi,\eta\in H_\sigma$ with
$\|\xi\|_{H_\sigma}=\|\eta\|_{H_\sigma}=1$, we have
\begin{align*}
\|\hat f(\sigma)(\xi,\eta)\|_\mathcal{H}&=\Big\|\int_G\langle\sigma(x)^*\xi,\eta\rangle_{H_\sigma}\,
f(x)\,dm_G(x)\Big\|_\mathcal{H}\\
&\le \int_G|\langle\sigma(x)^*\xi,\eta\rangle_{H_\sigma}|
\|f(x)\|_\mathcal{H}dm_G(x)\\
&\le\int_G\|f(x)\|_\mathcal{H}\,dm_G(x)\\
&=\|f\|_{L^1(G,\mathcal{H})},
\end{align*}
using meanwhile the Cauchy-Schwarz inequality and the fact that  $\sigma(x)^*$ is unitary to obtain $|\langle\sigma(x)^*\xi,\eta\rangle_{H_\sigma}|\le1$.
Hence, 
\begin{equation}\label{OneInfinity}
\|\widehat f\|_{\mathscr{S}_\infty(\widehat{G},\mathcal{H})}\le\|f\|_{L^1(G,\mathcal{H})}. 
\end{equation}
Moreover, by Theorem~\ref{theo:plancherel}, we have 
\begin{equation}
\|\widehat f\|_{\mathscr{S}_2(\widehat{G},\mathcal{H})}=\|f\|_{L^2(G,\mathcal{H})}. 
\end{equation}

Now, using  \cite[Theorem 5.1.2]{BerghLofstrom} related to the  interpolation of  Bochner $L^p$-spaces,  we obtain  that if $1\leq p\leq 2$, then  
$$
\|\widehat f\|_{\mathscr{S}_{p'}(\widehat{G},\mathcal{H})}\le\|f\|_{L^p(G,\mathcal{H})},
$$
after noting that  each  space  $\mathscr{S}_p(\widehat{G},\mathcal{H})$ carries a  counting measure 
weighted by $d_\sigma$ on the discrete index set $\{(\sigma,i,j):\sigma\in \widehat{G}, 1\le i,j\le d_\sigma\}$. 
\end{proof}

We now prove the analogue of the inverse Hausdorff-Young inequality.
\begin{theorem}\label{theo:invHY}
Let $G$ be a compact group,  $\mathcal{H}$  a complex Hilbert space  and $1\le p\le2$. The inversion map $\varphi\mapsto\varphi^\vee$, defined by
\begin{equation}
\varphi^\vee(x)=\sum_{\sigma\in\widehat{G}}d_\sigma\sum_{i=1}^{d_\sigma}\sum_{j=1}^{d_\sigma}u^\sigma_{ij}(x)\varphi(\sigma)_{ij},\quad x\in G,
\end{equation}
extends to a bounded linear operator $\mathscr{S}_p(\widehat{G},\mathcal{H})\to L^{p'}(G,\mathcal{H})$, with
\[
\|\varphi^\vee\|_{L^{p'}(G,\mathcal{H})}\le\|\varphi\|_{\mathscr{S}_p(\widehat{G},\mathcal{H})},\qquad
\frac1p+\frac1{p'}=1.
\]
\end{theorem}

\begin{proof}
 Let $\varphi\in\mathscr
S_1(\widehat{G},\mathcal{H})$. Then,  $\sum\limits_{\sigma\in\widehat{G}}d_\sigma\sum\limits_{i=1}^{d_\sigma}\sum\limits_{j=1}^{d_\sigma}\|\varphi(\sigma)_{ij}\|_\mathcal{H}<\infty$.
We have   $|u^\sigma_{ij}(x)|=|\langle\sigma(x)\xi_i^\sigma,
\xi_j^\sigma\rangle_{H_\sigma}|\le\|\sigma(x)\|_{B(H_\sigma)}\|\xi_i^\sigma\|_{H_\sigma}\|\xi_j^\sigma\|_{H_\sigma}=1$
for every $x\in G$, using  the Cauchy-Schwarz inequality, the unitarity of $\sigma(x)$ and the fact that  $\|\xi_j^\sigma\|_\sigma=\|\xi_i^\sigma\|_\sigma
=1$. Hence, for every $x$,
\begin{align*}
\sum_{\sigma\in\widehat{G}}d_\sigma\sum_{i=1}^{d_\sigma}\sum_{j=1}^{d_\sigma}\big\|u^\sigma_{ij}(x)\varphi(\sigma)_{ij}\big\|_\mathcal{H}&\le
\sum_{\sigma\in\widehat{G}}d_\sigma\sum_{i=1}^{d_\sigma}\sum_{j=1}^{d_\sigma}\|\varphi(\sigma)_{ij}\|_\mathcal{H}\\
&=\|\varphi\|_{\mathscr{S}_1(\widehat{G},\mathcal{H})}
<\infty,
\end{align*}
independently of $x$. Since $\mathcal{H}$ is complete, this shows that the series defining
$\varphi^\vee(x)$ converges absolutely in $\mathcal{H}$  and the convergence is
uniform in $x$. A uniform limit of the
partial sums, each being a finite sum of continuous $\mathcal{H}$-valued functions $u^\sigma_{ij}$, is itself continuous. So $\varphi^\vee$ is a well-defined element of $C(G,\mathcal{H})\subset
L^\infty(G,\mathcal{H})$, and
\begin{equation}\label{InfinityOne}
\|\varphi^\vee\|_{L^\infty(G,\mathcal{H})}=\sup_{x\in G}\|\varphi^\vee(x)\|_\mathcal{H}\le
\|\varphi\|_{\mathscr{S}_1(\widehat{G},\mathcal{H})}.
\end{equation}
Furthermore,  by Theorem~\ref{theo:plancherel}  the
map $f\mapsto\widehat f$ is an isometric isomorphism from $L^2(G,\mathcal{H})$ onto $\mathscr{S}_2(\widehat{G},\mathcal{H})$.
Its inverse is exactly the map $\varphi\mapsto\varphi^\vee$ above, restricted to $\mathscr{S}_2(\widehat{G},\mathcal{H})$. Hence
\begin{equation}\label{TwoTwo}
\|\varphi^\vee\|_{L^2(G,\mathcal{H})}=\|\varphi\|_{\mathscr{S}_2(\widehat{G},\mathcal{H})},\quad
\varphi\in\mathscr{S}_2(\widehat{G},\mathcal{H}).
\end{equation}

 By interpolating between (\ref{InfinityOne}) and (\ref{TwoTwo}) based on  \cite[Theorem 5.1.2]{BerghLofstrom},   
we obtain for $1\le p\le 2$,
$$
\|\varphi^\vee\|_{L^{p'}(G,\mathcal{H})}\le\|\varphi\|_{\mathscr{S}_p(\widehat{G},\mathcal{H})},\quad
\varphi\in\mathscr{S}_p(\widehat{G},\mathcal{H}).
$$
\end{proof}

For $K\subset L^p(G,\mathcal{H})$,  set 
$\widehat K=\{\widehat f:f\in K\}$.
Next, we present the two primary definitions of this article.
\begin{definition}\label{def:equicont}
The set $K\subset L^p(G,\mathcal{H})$ is  said to be uniformly $L^p(G,\mathcal{H})$-equicontinuous if for every
$\varepsilon>0$ there is an open neighborhood $O$ of $e$ such that $$\|R_yf-f\|_{L^p(G,\mathcal{H})}
<\varepsilon$$ for all $f\in K,\ y\in O$.
\end{definition}

\begin{definition}\label{def:decay}
The set $F\subset \mathscr{S}_p(\widehat{G},\mathcal{H})$ is said to have uniform $ \mathscr{S}_p(\widehat{G},\mathcal{H})$-decay if for every
$\varepsilon>0$ there is a finite set $S\subset\widehat{G}$ such that  $$\|\varphi\|_{\mathscr{S}_p(\widehat{G}\setminus
S,\mathcal{H})}<\varepsilon$$ for all $\varphi\in F$.
\end{definition}

\bigskip

For $f\in L^p(G,\mathcal{H})$ and for the complex-valued function $k\in L^1(G)$, the convolution $f*k$ is defined by 
\begin{equation}\label{Convolution}
(f*k)(x)=\int_G k(y)f(xy^{-1})
dm_G(y).
\end{equation}
 In Fourier domain, we have, for $\sigma\in\widehat{G}$,
\begin{equation}\label{ConvolutionFourierDomain}
\widehat{f*k}(\sigma)=\widehat k(\sigma)\widehat f(\sigma).
\end{equation}
In coordinates,  Equality~\ref{ConvolutionFourierDomain} turns to $$\widehat{f*k}(\sigma)_{ij}=\sum_{\ell=1}^{d_\sigma}\widehat k(\sigma)_{i\ell}
\,\widehat{f}(\sigma)_{\ell j}.$$
This is an ordinary matrix product of the scalar
matrix $\widehat k(\sigma)$ by the $\mathcal{H}$-valued matrix $\widehat{f}(\sigma)$, with
matrix multiplication acting entrywise via the vector space structure of $\mathcal{H}$. 

\bigskip

\begin{lemma}\label{lem:decay}
Let $G$ be a compact group and $\mathcal{H}$ a complex Hilbert space. Let $K\subset L^2(G,\mathcal{H})$ be uniformly $L^2(G,\mathcal{H})$-equicontinuous. Then $\widehat K$
has uniform $\mathscr{S}_2(\widehat{G},\mathcal{H})$-decay.
\end{lemma}

\begin{proof}
Let $\varepsilon>0$. The uniform $L^2(G,\mathcal{H})$-equicontinuity of $K$ implies the existence of an open neighborhood $U$ of $e$ such that $\|R_yf-f\|_{L^p(G,\mathcal{H})}
<\varepsilon$ for $f\in K,\ y\in U$.  
 Set $O=U\cap U^{-1}$. Then, $O$
is a symmetric open neighborhood  of $e$ 
such that $\|R_yf-f\|_{L^2(G,\mathcal{H})}<\varepsilon$ for $f\in K,\ y\in O$.  Set $k=\displaystyle\frac{1}{m_G(O)}1_O$ where $1_O$ designates the characteristic function of $O$. We have 
 $k\ge0$, $k(y^{-1})=k(y), \, \forall y\in O$, $\displaystyle\int_Gk\,dm_G=1$ and 
$k\in  L^2(G)\subset L^1(G)$ (the inclusion is due to the fact that $m_G$ is a finite measure).

Let $f\in K$. We have  $$(f*k)(x)-f(x)=\int_O\big(R_{y^{-1}}f(x)-f(x)\big)k(y)\,dm_G(y).$$ 
Therefore, 
\begin{align*}
\|f*k-f\|_{L^2(G,\mathcal{H})}&=\left(\int_G\left\|\int_O\big(R_{y^{-1}}f(x)-f(x)\big)k(y)\,dm_G(y)\right\|_{\mathcal{H}}^2dm_G(x)\right)^{\frac{1}{2}}\\
&\le \left[\int_G\left(\int_O k(y)\left\|\big(R_{y^{-1}}f(x)-f(x)\big)\right\|_{\mathcal{H}}dm_G(y)\right)^2dm_G(x)\right]^{\frac{1}{2}}\\
&\leq \int_O k(y) \left[\int_G \left\|\big(R_{y^{-1}}f(x)-f(x)\big)\right\|_{\mathcal{H}}^2 dm_G(x)\right]^{\frac{1}{2}} dm_G(y)\\
&\text{(by the the Minkowski's integral inequality)}\\
&=\int_Ok(y)\,\|R_{y^{-1}}f-f\|_{L^2(G,\mathcal{H})}\,dm_G(y)\\
&<\varepsilon \int_Gk(y)\,dm_G(y)=\varepsilon.
\end{align*}

Then, by Theorem~\ref{theo:plancherel}, 
 applied to $f*k-f$, we have
$$
\|\widehat{f*k}-\widehat f\|_{\mathscr{S}_2(\widehat{G},\mathcal{H})}=\|f*k-f\|_{L^2(G,\mathcal{H})}<\varepsilon,\quad f\in K.
$$
 Fix $\sigma$ in $\widehat{G}$. For each $i,j$, using the triangle inequality in $\mathcal{H}$ and the  Cauchy-Schwarz inequality on the sum over $\ell$, we successively have
 \begin{align*}
\big\|\widehat{f*k}(\sigma)_{ij}\big\|_\mathcal{H}&=\Big\|\sum_{\ell=1}^{d_\sigma}\widehat k(\sigma)_{i\ell}
\widehat{f}(\sigma)_{\ell j}\Big\|_\mathcal{H}\\
&\leq \sum_{\ell=1}^{d_\sigma}|\widehat k(\sigma)_{i\ell}|
\left\|\widehat{f}(\sigma)_{\ell j}\right\|_\mathcal{H}\\
&\le\Big(\sum_{\ell=1}^{d_\sigma}|\widehat k(\sigma)_{i\ell}|^2\Big)^{1/2}
\Big(\sum_{\ell=1}^{d_\sigma}\|\widehat{f}(\sigma)_{\ell j}\|_\mathcal{H}^2\Big)^{1/2}.
\end{align*}
Squaring and  summing over $i,j$  yield:
\begin{equation}\label{Path}
\sum_{i=1}^{d_\sigma}\sum_{j=1}^{d_\sigma}\big\|\widehat{f*k}(\sigma)_{ij}\big\|_\mathcal{H}^2\le\Big(\sum_{i=1}^{d_\sigma}\sum_{j=1}^{d_\sigma}
\|\widehat{f}(\sigma)_{ij}\|_\mathcal{H}^2\Big)\cdot\|\widehat k(\sigma)\|_{B_2(H_\sigma)}^2,
\end{equation}
 
\noindent Furthermore, since $k\in L^2(G)$, the (scalar) Plancherel theorem (Theorem~\ref{scalar:plancherel}) gives
$$\sum_{\sigma\in \widehat{G}}
d_\sigma\|\widehat k(\sigma)\|_{B_2(H_\sigma)}^2=\|k\|_{L^2(G)}^2=1<\infty.$$

Since the series converges, its general term must tend to zero. So,  there is a finite set $S\subset\widehat{G}$ such that $d_\sigma\|\widehat k(\sigma)\|_{B_2(H_\sigma)}^2\le\frac{1}{4}$ for $\sigma\notin S$. This yields $\|\widehat k(\sigma)\|_{B_2(H_\sigma)}^2\le \frac{1}{4}$ since $d_\sigma\ge 1$. 

Multiplying (\ref{Path}) by
$d_\sigma$ and summing over $\sigma\notin S$ yield:
$$
\|\widehat{f*k}\|_{\mathscr{S}_2(\widehat{G}\setminus S,\mathcal{H})}\le\tfrac12\|\widehat f\|_{\mathscr{S}_2(\widehat{G}\setminus S,\mathcal{H})}.$$
Now, 
\begin{align*}
\|\widehat f\|_{\mathscr{S}_2(\widehat{G}\setminus S,\mathcal{H})}&=\|\widehat{f}-\widehat{f*k}+\widehat{f*k}\|_{\mathscr{S}_2(\widehat{G}\setminus S,\mathcal{H})}\\
&\le\|\widehat f-\widehat{f*k}\|_{\mathscr{S}_2(\widehat{G}\setminus S,\mathcal{H})}
+\|\widehat{f*k}\|_{\mathscr{S}_2(\widehat{G}\setminus S,\mathcal{H})}\\
&<\varepsilon +\frac{1}{2}\|\widehat f\|_{\mathscr{S}_2(\widehat{G}
\setminus S,\mathcal{H})}.
\end{align*}
Thus,  $\|\widehat f\|_{\mathscr{S}_2(\widehat{G}\setminus S,\mathcal{H})}<2\varepsilon$ for every $f\in K$, with $S$
independent of $f$. This is the uniform $\mathscr{S}_2(G,\mathcal{H})$-decay of $\widehat K$.
\end{proof}

\begin{lemma}\label{lem:equicont}
Let $G$ be a compact group and $\mathcal{H}$ a complex Hilbert space.  Let $1\le p\le 2$ and  $p'$ be such that $\frac{1}{p}+\frac{1}{p'}=1$.   Let $K\subset L^{p'}(G,\mathcal{H})$ be bounded,
and suppose that $\widehat K$ has uniform $\mathscr{S}_p(\widehat{G},\mathcal{H})$-decay. Then $K$ is uniformly
$L^{p'}(G,\mathcal{H})$-equicontinuous.
\end{lemma}

\begin{proof}
Let $\varepsilon>0$. By uniform $\mathscr{S}_p(\widehat{G},\mathcal{H})$-decay of
$\widehat K$, there is a finite set $S\subset\widehat{G}$ such that
$$
\|\widehat f\|_{\mathscr{S}_p(\widehat{G}\setminus S,\mathcal{H})}<\frac{\varepsilon}{3},\quad f\in K.
$$
Let $P_S$ be the band-limiting projection defined by
$$\widehat{P_Sf}(\sigma)=\left\lbrace\begin{array}{cc}
\widehat f(\sigma), & \text{ if } \sigma\in S\\ 
0, & \text{ otherwise. }
\end{array}\right.$$ 
 Equivalently $f-P_Sf$ has Fourier transform
$\widehat f\cdot 1_{\widehat{G}\setminus S}$ where $1_{\widehat{G}\setminus S}$ is the characteristic function of $\widehat{G}\setminus S$. By Theorem~\ref{theo:invHY} applied to $\varphi=\widehat f\cdot\mathbf1_{\widehat{G}\setminus S}
\in\mathscr{S}_p(\widehat{G},\mathcal{H})$, whose inverse Fourier transform is $f-P_Sf$, we obtain
\begin{equation}\label{RLL}
\|f-P_Sf\|_{L^{p'}(G,\mathcal{H})}\le\|\widehat f\|_{\mathscr{S}_p(\widehat{G}\setminus S,\mathcal{H})}<\frac{\varepsilon}{3},
\quad \forall f\in K. 
\end{equation}

The projection $P_S$ commutes with translation. Indeed, for $y\in G$, $f\in L^{p'}(G,\mathcal{H})$,
$\sigma\in\widehat{G}$, $\xi,\eta\in H_\sigma$,  we have

\begin{align*}
\widehat{R_yf}(\sigma)(\xi,\eta)&=\int_G\langle\sigma(x)^*\xi,\eta\rangle_{H_\sigma}\,
f(xy)\,dm_G(x)\\
&=\int_G\langle\sigma(zy^{-1})^*\xi,\eta\rangle_{H_\sigma}f(z)\,dm_G(z)\\ 
&\text{(by the change of variable }   x=zy^{-1} \\
&\text{ and the invariance of the Haar measure})\\
&=\int_G\langle\sigma(y)\sigma(z)^*\xi,\eta\rangle_{H_\sigma}\,f(z)\,dm_G(z)\\
&=\int_G\langle\sigma(z)^*\xi,\sigma(y)^*\eta\rangle_{H_\sigma}\,f(z)\,dm_G(z)\\
&=\hat f(\sigma)\big(\xi,\sigma(y)^*\eta\big),
\end{align*}
 In particular $\widehat{R_yf}$ is
supported on the same set of $\sigma$'s as $\widehat f$.
Hence,
$$
P_S(R_yf)=R_y(P_Sf),\qquad f\in L^{p'}(G,\mathcal{H}),\ y\in G. 
$$
Now, let us show that $P_Sf$ is continuous, and its finitely many coefficients are uniformly
bounded on $K$. By the inversion formula, restricted to the finite set $S$, we have
$$
P_Sf(x)=\sum_{\sigma\in S}d_\sigma\sum_{i=1}^{d_\sigma}\sum_{j=1}^{d_\sigma}\widehat f(\sigma)_{ij}\,
u^\sigma_{ij}(x),
$$
which is a finite sum of continuous $\mathcal{H}$-valued functions, hence continuous on $G$; by
the fact that a continuous vector-valued function on a compact space is bounded, each term of $P_Sf$, and hence $P_Sf$, is
bounded on $G$. Moreover, by  the H\"older's inequality applied to
$\widehat f(\sigma)_{ij}=\displaystyle\int_G\overline{u^\sigma_{ij}(x)}\,f(x)\,dm_G(x)$ against
$u^\sigma_{ij}\in L^p(G)$ (recall $|u^\sigma_{ij}|\le1$, so $u^\sigma_{ij}\in
L^p(G)$ for every $p$, with $\|u^\sigma_{ij}\|_{L^p(G)}\le1$), we have 
$$
\|\widehat f(\sigma)_{ij}\|_\mathcal{H}\le\|u^\sigma_{ij}\|_{L^p(G)}\,\|f\|_{L^{p'}(G,\mathcal{H})}\le
\|f\|_{L^{p'}(G,\mathcal{H})},\quad\sigma\in S,\ f\in K.
$$
Since $K$ is bounded, say for all $f\in K$, $\|f\|_{L^{p'}(G,\mathcal{H})}\le M$ for some positive constant $M$,  then the coefficients set 
 $\{\widehat f(\sigma)_{ij}:\sigma\in S,\,1\le i,j\le
d_\sigma\}$ is uniformly bounded on $K$ by $M$; that is 

\begin{equation}\label{CoeffBound}
\|\widehat f(\sigma)_{ij}\|_\mathcal{H}\le M,\quad f\in K,\, \sigma\in S,\,1\le i,j\le
d_\sigma.
\end{equation}

Since $G$ is compact and each $u^\sigma_{ij}$ is continuous, $u^\sigma_{ij}$ is
uniformly continuous on $G$ by the Heine-Cantor theorem (a continuous function on a compact topological group, viewed as a compact uniform space, is uniformly continuous). As $S$
is finite, there is a single open neighborhood $O$ of $e$ such that, for every
$y\in O$, $\sigma\in S$, $1\le i,j\le d_\sigma$,
\begin{equation}\label{SupBound}
\sup_{x\in G}\big|u^\sigma_{ij}(xy)-u^\sigma_{ij}(x)\big|<
\frac{\varepsilon}{3M\sum\limits_{\sigma\in S}d_\sigma^3}.
\end{equation}
For $f\in K$ and $y\in O$, we have
\begin{align*}
\big\|R_y(P_Sf)(x)-P_Sf(x)\big\|_\mathcal{H}&=\Big\|\sum_{\sigma\in S}d_\sigma\sum_{i=1}^{d_\sigma}\sum_{j=1}^{d_\sigma}
\widehat f(\sigma)_{ij}\big(u^\sigma_{ij}(xy)-u^\sigma_{ij}(x)\big)\Big\|_\mathcal{H}\\
&\le \sum_{\sigma\in S}d_\sigma\sum_{i=1}^{d_\sigma}\sum_{j=1}^{d_\sigma}\Big\|\widehat f(\sigma)_{ij}\Big\|_\mathcal{H}\big |u^\sigma_{ij}(xy)-u^\sigma_{ij}(x)\big|\\
&<\frac{\varepsilon}{3}, \text{ using (\ref{CoeffBound}) and (\ref{SupBound}).}
\end{align*}
 Integrating  the $p'$-th power of both sides over $G$, we obtain
\begin{equation}\label{RL}
\|R_y(P_Sf)-P_Sf\|_{L^{p'}(G,\mathcal{H})}<\frac{\varepsilon}{3},\quad f\in K,\ y\in O. 
\end{equation}

  We break $R_yf-f$ down into three parts:
$$
R_yf-f=R_y(f-P_Sf)+\big(R_y(P_Sf)-P_Sf\big)-(f-P_Sf).
$$
Thus,  using the fact that right translation is an isometry of $L^{p'}(G,\mathcal{H})$
 (due to the right-invariance of $m_G$) and  the inequalities  (\ref{RLL}) and   (\ref{RL}), we obtain
\begin{align*}
\|R_yf-f\|_{L^{p'}(G,\mathcal{H})}&\le\|R_y(f-P_Sf)\|_{L^{p'}(G,\mathcal{H})}
+\|R_y(P_Sf)-P_Sf\|_{L^{p'}(G,\mathcal{H})}+\|P_Sf-f\|_{L^{p'}(G,\mathcal{H})}\\
&<\frac{\varepsilon}{3}+\frac{\varepsilon}{3}+\frac{\varepsilon}{3}=\varepsilon.
\end{align*}
Hence, $K$ is uniformly
$L^{p'}(G,\mathcal{H})$-equicontinuous.
\end{proof}

\begin{corollary}\label{cor:equiv}
Let $G$ be a compact group and $\mathcal{H}$ a complex Hilbert space. Let $K\subset L^2(G,\mathcal{H})$ be bounded. Then,  $K$ is uniformly $L^2(G,\mathcal{H})$-equicontinuous if and only if
$\widehat K$ has uniform $\mathscr{S}_2(G,\mathcal{H})$-decay.
\end{corollary}

\begin{proof}
 Combining Lemma~\ref{lem:decay}  and  Lemma~\ref{lem:equicont} (restricted  to the case $p=p'=2$)  yields the desired statement.
\end{proof}

\begin{lemma}\label{lemma:translation-continuous}
Let $G$ be a compact group, $\mathcal{H}$ a complex Hilbert space, and $f\in L^2(G,\mathcal{H})$.
Then, the map $G\to L^2(G,\mathcal{H})$, $y\mapsto R_yf$ is continuous. 
\end{lemma}

\begin{proof}
By the construction of the Bochner integral,  continuous $\mathcal{H}$-valued functions are dense
in $L^2(G,\mathcal{H})$. So, given $\delta>0$, choose $g\in C(G,\mathcal{H})$ such that
\begin{equation}\label{DI}
\|f-g\|_{L^2(G,\mathcal{H})}<\delta.
\end{equation}

Since  $R_y$ is an isometry of $L^2(G,\mathcal{H})$ for every $y\in G$,  we have
\begin{equation}\label{DDI}
\|R_yf-R_yg\|_{L^2(G,\mathcal{H})}=\|R_y(f-g)\|_{L^2(G,\mathcal{H})}=\|f-g\|_{L^2(G,\mathcal{H})}<\delta.
\end{equation}

Since $G$ is compact and $g$ is continuous, $g$ is uniformly continuous. Thus, for every
$\eta>0$ there is an open neighborhood $V$ of $e$ such that
$$
\|g(xy)-g(x)\|_\mathcal{H}<\eta,\qquad x\in G,\ y\in V.
$$
 Let  $y\in V$. We have
$$
\|R_yg-g\|_{L^2(G,\mathcal{H})}^2=\int_G\|g(xy)-g(x)\|_\mathcal{H}^2\,dm_G(x)\le\eta^2\,m_G(G)=\eta^2.
$$
Thus, 
\begin{equation}\label{DDDI}
\|R_yg-g\|_{L^2(G,\mathcal{H})}\le\eta,\quad
y\in V.
\end{equation}
Now,  given $\varepsilon>0$, first
apply  (\ref{DI}) with $\delta=\frac{\varepsilon}{3}$ to get $g\in C(G,\mathcal{H})$ such that
$\|f-g\|_{L^2(G,\mathcal{H})}<\frac{\varepsilon}{3}$. Then apply (\ref{DDDI}) with $\eta=\frac{\varepsilon}{3}$ to
get an open neighborhood $V$ of $e$ such that $\|R_yg-g\|_{L^2(G,\mathcal{H})}\le\frac{\varepsilon}{3}$. We now get
\begin{align*}
\|R_yf-f\|_{L^2(G,\mathcal{H})}&\le\|R_yf-R_yg\|_{L^2(G,\mathcal{H})}+\|R_yg-g\|_{L^2(G,\mathcal{H})}+
\|g-f\|_{L^2(G,\mathcal{H})}\\
&<\frac{\varepsilon}{3}+\frac{\varepsilon}{3}+\frac{\varepsilon}{3}=\varepsilon.
\end{align*}
by using (\ref{DDI}) for the first and third terms.
Thus, the map $y\mapsto R_yf$ is continuous at $e$. To check the continuity at any point $y_0\in G$, write
$$R_yf-R_{y_0}f=R_{y_0}\big(R_{y_0^{-1}y}f-f\big)$$
and  use the fact that $\|\cdot\|_{L^2(G,\mathcal{H})}$
is unaffected by the isometry $R_{y_0}$ togheter with the continuity at $e$. The conclusion follows.
\end{proof}

\begin{lemma}\label{lem:PS-finite-dim}
Let $G$ be a compact group, $\mathcal{H}$ a complex Hilbert space, and $S$  a finite subset of $\widehat{G}$.
Let $P_S:L^2(G,\mathcal{H})\to L^2(G,\mathcal{H})$ be the band-limiting projection defined by
$$\widehat{P_Sf}(\sigma)=\left\lbrace\begin{array}{cc}
\widehat f(\sigma), & \text{ if } \sigma\in S\\ 
0, & \text{ otherwise. }
\end{array}\right.$$ 
 Consider the finite number $
N=\sum\limits_{\sigma\in S}d_\sigma^2$.

Then, the range $P_S\big(L^2(G,\mathcal{H})\big)$ is
linearly isometric (up to a fixed rescaling of the norm) to $\mathcal{H}^N$,  the
direct sum of $N$ copies of $\mathcal{H}$.
\end{lemma}

\begin{proof}
 Set $\mathcal{I}=\{(\sigma,i,j): \sigma\in S, 1\le i,j\le d_\sigma\}$. The cardinality of $\mathcal{I}$ is $N$. Enumerate the  triples
$(\sigma,i,j)$ in  $\mathcal{I}$  as $1,\dots,N$. Define
$$
\Psi:\mathcal{H}^N\longrightarrow L^2(G,\mathcal{H}),\quad
\Psi\big((a_{\sigma,i,j})_{(\sigma,i,j)\in \mathcal{I}} \big)=\sum_{\sigma\in S}d_\sigma
\sum_{i=1}^{d_\sigma}\sum_{j=1}^{d_\sigma}u^\sigma_{ij}(\,\cdot\,)a_{\sigma,i,j}.
$$
The map $\Psi$ is manifestly linear, and well-defined into $L^2(G,\mathcal{H})$ since it is a finite sum
of continuous, hence $L^2(G,\mathcal{H})$-valued functions.

\noindent The image of $\Psi$ is exactly $P_S\big(L^2(G,\mathcal{H})\big)$. Indeed, if $g=\Psi(a)$
for some $a\in\mathcal{H}^N$, then
 $\widehat{g}=\widehat{g}\cdot 1_S$; that is  $g=P_Sg$. So,  the image of $\Psi$ lies in
$P_S\big(L^2(G,\mathcal{H})\big)$. Conversely, for any $f\in L^2(G,\mathcal{H})$, setting
$a_{\sigma,i,j}=\hat f(\sigma)_{ij}$ for $\sigma\in S$ gives, by the inversion formula
restricted to $S$,
$$
\Psi(a)=\sum_{\sigma\in S}d_\sigma\sum_{i=1}^{d_\sigma}\sum_{j=1}^{d_\sigma}u^\sigma_{ij}\widehat f(\sigma)_{ij}=P_Sf,
$$
so every element of $P_S\big(L^2(G,\mathcal{H})\big)$ is of the form $\Psi(a)$. Hence,
$\Psi\big(\mathcal{H}^N\big)=P_S\big(L^2(G,\mathcal{H})\big)$.

Let us prove that  $\Psi$ is an  isometry.  By  Theorem~\ref{theo:plancherel} applied to $g=\Psi(a)$, we have 
$$
\|\Psi(a)\|_{L^2(G,\mathcal{H})}^2=\|\widehat{g}\|_{\mathscr{S}_2(G,\mathcal{H})}^2=\sum_{\sigma\in S}
d_\sigma\sum_{i,j}\|a_{\sigma,ij}\|_\mathcal{H}^2.
$$
Define the norm of  $a=(a_{\sigma,i,j})_{(\sigma,i,j)\in \mathcal{I}}\in \mathcal{H}^N$  by:
$$\|a\|_{w}^2=\sum_{\sigma\in S}d_\sigma\sum_{i=1}^{d_\sigma}\sum_{j=1}^{d_\sigma}\|a_{\sigma,ij}\|_\mathcal{H}^2.$$ 

So
\[
\|\Psi(a)\|_{L^2(G,\mathcal{H})}=\|a\|_w,\qquad a\in\mathcal{H}^N,
\]
i.e.\ $\Psi$ is an isometry from $(\mathcal{H}^N,\|\cdot\|_w)$ onto $P_S\big(L^2(G,\mathcal{H})\big)$
with its $L^2(G,\mathcal{H})$-norm. In particular $\Psi$ is injective (isometries are
injective). In conclusion, 
 $P_S\big(L^2(G,\mathcal{H})\big)$ is linearly isometric, up to this fixed rescaling, to
$\mathcal{H}^N$.
\end{proof}

The trivial representation of $G$ is the representation $\sigma_0$ given by  $\sigma_0(x)w=w$ for all $x\in G, \,w\in \mathbb C$.

\begin{proposition}\label{Claim}
Let $\mathcal{H}$ be a complex Hilbert space,  $a\in \mathcal{H}$, and $f\in L^2(G,\mathcal{H})$ the constant function $f(x)\equiv a$. Then
$\widehat f(\sigma)=0$ for every $\sigma\in\widehat{G}\setminus\{\sigma_0\}$.
\end{proposition}

\begin{proof}
Fix  $\sigma\in\widehat{G}$ with $\sigma\ne\sigma_0$. By definition, for $\xi,\eta\in H_\sigma$,  we have
$$
\widehat f(\sigma)(\xi,\eta)=\int_G\langle\sigma(x)^*\xi,\eta\rangle_{H_\sigma}\,f(x)\,
dm_G(x)=\Big(\int_G\langle\sigma(x)^*\xi,\eta\rangle_{H_\sigma}\,dm_G(x)\Big)\,a.
$$
 It remains to show that scalar integral  $\displaystyle\int_G\langle\sigma(x)^*\xi,
\eta\rangle_{H_\sigma}\,dm_G(x)$ vanishes.

By definition of the scalar Fourier transform
on $G$ applied to the constant function $\mathbf 1(x)\equiv1\in L^1(G)$, we have 
$$
\widehat{\mathbf 1}(\sigma)=\int_G\mathbf 1(x)\,\sigma(x)^*\,dm_G(x)=\int_G\sigma(x)^*\,
dm_G(x),
$$
so $\displaystyle\int_G\langle\sigma(x)^*\xi,\eta\rangle_{H_\sigma}\,dm_G(x)=\langle\widehat{\mathbf1}(\sigma)
\xi,\eta\rangle_{H_\sigma}$, and it suffices to show that $\widehat{\mathbf 1}(\sigma)=0$.

 For any fixed
$y\in G$, we have
\begin{align*}
\sigma(y)\,\hat{\mathbf1}(\sigma)&=\sigma(y)\int_G\sigma(x)^*\,dm_G(x)\\
&=\int_G
\sigma(y)\sigma(x)^*\,dm_G(x)\\
&=\int_G\sigma(y x^{-1})\,dm_G(x)\\
&(\text{by the invariance of } m_G)\\
&=\int_G\sigma(x^{-1})\,dm_G(x)\\
&=\int_G\sigma(x)^*\,dm_G(x)=\widehat{\mathbf 1}(\sigma).
\end{align*}
A similar  computation  gives
$\widehat{\mathbf1}(\sigma)\sigma(y)=\hat{\mathbf1}(\sigma)$. Therefore, 
$\widehat{\mathbf1}(\sigma)$ commutes with $\sigma(y)$ for every $y\in G$. 

 Since $\sigma$ is irreducible, Schur's lemma applied
to the  operator $\widehat{\mathbf1}(\sigma)$  gives
$\widehat{\mathbf1}(\sigma)=c\cdot\mathrm{Id}_{H_\sigma}$ for some scalar $c\in\mathbb C$.
But $\hat{\mathbf 1}(\sigma)=\sigma(y)\hat{\mathbf1}(\sigma)$ for every $y$, thus 
 $c\cdot\mathrm{Id_{H_\sigma}}=c\cdot\sigma(y)$ for every $y\in G$.  If $c\ne0$ this would
give $\sigma(y)=\mathrm{Id}_{H_\sigma}$ for every $y\in G$, i.e.\ $\sigma=\sigma_0$
which contradicts $\sigma\ne\sigma_0$. Hence $c=0$, i.e.
$$
\widehat{\mathbf1}(\sigma)=0,\quad\sigma\in\widehat{G}\setminus\{\sigma_0\}.
$$

\end{proof}

Hereafter is a vector version of the Pego theorem on compact groups (Version 1). 
 
\begin{theorem}\label{thm:main}
Let $G$ be a compact group and  $\mathcal{H}$ a finite-dimensional complex Hilbert space. Let  $K$ be a bounded subset of
$L^2(G,\mathcal{H})$. The following assertions are equivalent: 
\begin{enumerate}
\item[(i)] $K$ is precompact;
\item[(ii)] $K$ is uniformly $L^2(G,\mathcal{H})$-equicontinuous;
\item[(iii)] $\widehat K$ has uniform
$\mathscr{S}_2(\widehat{G},\mathcal{H})$-decay.
\end{enumerate}
\end{theorem}

\begin{proof}
By Corollary~\ref{cor:equiv}, assertions (ii) and (iii) are already known to be equivalent. It therefore suffices to prove
(i)$\Rightarrow$(ii) and (iii)$\Rightarrow$(i) to close the cycle and establish the equivalence of all the three assertions.

\medskip
\noindent\textbf{(i) $\Rightarrow$ (ii).}
Let $\varepsilon>0$. Since $K$ is precompact, it is totally bounded, so there exist
$f_1,\dots,f_m\in L^2(G,\mathcal{H})$ such that
$$
K\subset\bigcup_{i=1}^m B\left(f_i,\frac{\varepsilon}{3}\right),
$$
where $B\left(g,r\right)$ is the open ball of centre $g$ and radius $r$.
For $i\in\{1,\dots,m\}$, the map $y\mapsto R_yf_i$ is continuous from $G$ into
$L^2(G,\mathcal{H})$ (Lemma~\ref{lemma:translation-continuous}). Hence,  there is an open neighborhood $O_i$ of $e$ such that
$\|R_yf_i-f_i\|_{L^2(G,\mathcal{H})}<\frac{\varepsilon}{3}$ 
for $y\in O_i$. Let $O= \inter\limits_{i=1}^mO_i$,
an open neighborhood of $e$.

Let $f\in K$ and $y\in O$. Choose $i$ such that $\|f-f_i\|_{L^2(G,\mathcal{H})}<\frac{\varepsilon}{3}$.
Since $R_y$ is an isometry of $L^2(G,\mathcal{H})$,
$\|R_yf-R_yf_i\|_{L^2(G,\mathcal{H})}=\|f-f_i\|_{L^2(G,\mathcal{H})}<\frac{\varepsilon}{3}$. So, by the triangle
inequality
$$
\|R_yf-f\|_{L^2(G,\mathcal{H})}\le\|R_yf-R_yf_i\|_{L^2(G,\mathcal{H})}+\|R_yf_i-f_i\|_{L^2(G,\mathcal{H})}+
\|f_i-f\|_{L^2(G,\mathcal{H})}<\varepsilon.
$$
Therefore,
$K$ is uniformly $L^2(G,\mathcal{H})$-equicontinuous.

\medskip
\noindent\textbf{(iii) $\Rightarrow$ (i).}
Let $\varepsilon>0$. By (iii), choose a finite set $S\subset\widehat{G}$ with
$\|\widehat f\|_{\mathscr{S}_2(\widehat{G}\setminus S,\mathcal{H})}<\frac{\varepsilon}{2}$ for all $f\in K$. Let
$P_S$ be the band-limiting projection defined in  Lemma~\ref{lem:PS-finite-dim}.  By Theorem~\ref{theo:plancherel}, we have
$$
\|f-P_Sf\|_{L^2(G,\mathcal{H})}=\|\widehat f-\widehat{P_Sf}\|_{\mathscr{S}_2(\widehat{G},\mathcal{H})}=
\|\widehat f\|_{\mathscr{S}_2(\widehat{G}\setminus S,\mathcal{H})}<\frac{\varepsilon}{2},\quad f\in K.
$$

\noindent By Lemma~\ref{lem:PS-finite-dim}, the
range $P_S\big(L^2(G,\mathcal{H})\big)$ is isomorphic to a finite direct sum of copies of $\mathcal{H}$.  Since $\dim\mathcal{H}<\infty$ by hypothesis, this finite direct sum is finite-dimensional. Moreover, the set $\{P_Sf:f\in K\}$ is a bounded subset of this finite-dimensional space because $P_S$
is an orthogonal projection on $L^2(G,\mathcal{H})$, so $\|P_Sf\|_{L^2(G,\mathcal{H})}\le
\|f\|_{L^2(G,\mathcal{H})}$, and $K$ is bounded by hypothesis. A bounded subset of a
finite-dimensional normed space is precompact according to the Bolzano-Weierstrass theorem.  Thus, 
$\{P_Sf:f\in K\}$ is totally bounded: there exists $\{g_1,\dots,g_k\}\subset
P_S\big(L^2(G,\mathcal{H})\big)$ such that 
$$\{P_Sf:f\in K\}\subset \union\limits_{j=1}^k B\left(g_j,\frac{\varepsilon}{2}\right).$$

For each $f\in K$, there exists  $g_j\in \{g_1,\dots,g_k\}$ such that $\|P_Sf-g_j\|_{L^2(G,\mathcal{H})}<\frac{\varepsilon}{2}$. By the
triangle inequality, 
\begin{align*}
\|f-g_j\|_{L^2(G,\mathcal{H})}&\le\|f-P_Sf\|_{L^2(G,\mathcal{H})}+\|P_Sf-g_j\|_{L^2(G,\mathcal{H})}\\
&<\frac{\varepsilon}{2}+\frac{\varepsilon}{2}=\varepsilon.
\end{align*}
Thus $K\subset\union\limits_{j=1}^kB(g_j,\varepsilon)$; that is,  $K$ is totally bounded in
$L^2(G,\mathcal{H})$. Hence, $K$ is precompact.
\end{proof}

\begin{remark}{\rm
When $\dim\mathcal{H}=1$, Theorem~\ref{thm:main} reduces to the Pego theorem proved by  Kumar (Theorem~\ref{thm:kumar}).
The finite-dimensionality of $\mathcal{H}$ is used only in the last paragraph of
(iii)$\Rightarrow$(i); it cannot be dropped without an additional hypothesis, since a
bounded set in an infinite-dimensional Hilbert space need not be precompact.
}\end{remark}

\begin{counterexample}\label{ex:counterexample}{\rm
Let $G$ be a compact group and  let $\mathcal{H}$ be an {\it infinite-dimensional} Hilbert space, with orthonormal sequence
$\{e_n\}_{n\ge1}$, and let $f_n\in L^2(G,\mathcal{H})$ be the constant function
$f_n(x)\equiv e_n,\quad x\in G$. Set $K=\{f_n : n\ge1\}$.

\begin{itemize}
\item $K$ is bounded: $\|f_n\|_{L^2(G,\mathcal{H})}=\|e_n\|_\mathcal{H}=1$ for every $n$.

\item $K$ is uniformly $L^2(G,\mathcal{H})$-equicontinuous: since $f_n$ is
constant, $R_yf_n=f_n$ for every $y\in G$. So,  Definition~\ref{def:equicont} holds with
$O=G$ for every $\varepsilon>0$.

\item $\widehat K$ has uniform $\mathscr{S}_2(\widehat{G},\mathcal{H})$-decay: a constant
function has Fourier transform supported only on the trivial representation
$\sigma_0\in\widehat{G}$ (see Proposition~\ref{Claim}), so $\widehat f_n(\sigma)=0$ for all $\sigma\ne\sigma_0$ and all $n$. Thus, 
$$
\|\widehat f_n\|_{\mathscr{S}_2(\widehat{G}\setminus\{\sigma_0\},\,\mathcal{H})}=0,\quad n\ge1,
$$
and Definition~\ref{def:decay} holds with the single finite set $S=\{\sigma_0\}$.

\item $K$ is not precompact: for $n\ne m$,
$$
\|f_n-f_m\|_{L^2(G,\mathcal{H})}=\|e_n-e_m\|_\mathcal{H}=\sqrt2.
$$
 Hence, $K$ has no Cauchy, and so no convergent, subsequence.
\end{itemize}
\noindent  Thus conditions (ii) and (iii) of Theorem~\ref{thm:main} both hold for $K$ while
condition (i) fails: the equivalence breaks down as soon as $\dim\mathcal{H}=\infty$.
}\end{counterexample}

\begin{remark}{\rm 
  When $\dim\mathcal{H}=\infty$, a bounded
set of values need not be precompact, and no amount of control in the $G$-direction
can repair that. Therefore an additional hypothesis on the  set of values in $\mathcal{H}$ would be
needed to restore the equivalence. To this end, we introduce the uniform thinness hypothesis to obtain an infinite-dimensional version of the Pego theorem. 
}\end{remark}

\begin{definition}\label{def:tight}
$K\subset L^2(G,\mathcal{H})$ is said to be uniformly tight if for every $\varepsilon>0$ there is a
finite-dimensional subspace $V\subset\mathcal{H}$ such that
$$
\|f - Q_V\circ f\|_{L^2(G,\mathcal{H})}<\varepsilon,\quad f\in K,
$$
where $Q_V:\mathcal{H}\to V$ is the orthogonal projection on $V$.
\end{definition}

Hereafter is a vector version of the Pego theorem on compact groups (Version 2).

\begin{theorem}\label{thm:main-tight}
Let $\mathcal{H}$ be a complex Hilbert space (not assumed finite-dimensional), and let
$K\subset L^2(G,\mathcal{H})$ be bounded and uniformly tight. Then, the following assertions are
equivalent:
\begin{enumerate}
\item[(i)] $K$ is precompact;
\item[(ii)] $K$ is uniformly
$L^2(G,\mathcal{H})$-equicontinuous;
\item[(iii)] $\widehat K$ has uniform $\mathscr{S}_2(\widehat{G},\mathcal{H})$-decay.
\end{enumerate} 
\end{theorem}

\begin{proof}
By Corollary~\ref{cor:equiv}, assertions (ii) and (iii) are already known to be equivalent. It
therefore suffices to prove (i)$\Rightarrow$(ii) and (iii)$\Rightarrow$(i).

\medskip
\noindent\textbf{(i) $\Rightarrow$ (ii).}
This implication does not use tightness or any dimension hypothesis on $\mathcal{H}$. Let $\varepsilon>0$. Since $K$ is precompact, it is totally bounded; that is,  there exist
$f_1,\dots,f_m\in L^2(G,\mathcal{H})$ with $K\subset\union\limits_{i=1}^mB(f_i,\frac{\varepsilon}{3})$. By
Lemma~\ref{lemma:translation-continuous}, for each $i$ the map $y\mapsto R_yf_i$
is continuous at $e$, so there is an open neighborhood $O_i$ of $e$ such that
$\|R_yf_i-f_i\|_{L^2(G,\mathcal{H})}<\frac{\varepsilon}{3}$ for $y\in O_i$. Let $O=\inter\limits_{i=1}^mO_i$.
For $f\in K$ and $y\in O$, there exists $i\in\{1,\cdots,m\}$ with $\|f-f_i\|_{L^2(G,\mathcal{H})}<\frac{\varepsilon}{3}$.
By the triangle inequality and the fact  that $R_y$ is an isometry of $L^2(G,\mathcal{H})$, we have
\begin{align*}
\|R_yf-f\|_{L^2(G,\mathcal{H})}&\le\|R_yf-R_yf_i\|_{L^2(G,\mathcal{H})}+\|R_yf_i-f_i\|_{L^2(G,\mathcal{H})}+
\|f_i-f\|_{L^2(G,\mathcal{H})}\\
&<\frac{\varepsilon}{3}+\frac{\varepsilon}{3}+\frac{\varepsilon}{3}=\varepsilon.
\end{align*}
Hence, $K$ is uniformly $L^2(G,\mathcal{H})$-equicontinuous.

\medskip
\noindent\textbf{(iii) $\Rightarrow$ (i).}
This is where boundedness alone is not enough when $\dim\mathcal{H}=\infty$
(Counter-example~\ref{ex:counterexample}), and where uniform tightness supplies exactly the missing control in $\mathcal{H}$.

Let $\varepsilon>0$.
 By hypothesis, there exists a  finite set 
$S\subset\widehat{G}$ with $\|\widehat f\|_{\mathscr{S}_2(\widehat{G}\setminus S,\mathcal{H})}<\frac{\varepsilon}{4}$ for
all $f\in K$. Let $P_S$ be the band-limiting projection. By Theorem~\ref{theo:plancherel}, we have 
$$
\|f-P_Sf\|_{L^2(G,\mathcal{H})}=\|\widehat f\|_{\mathscr{S}_2(\widehat{G}\setminus S,\mathcal{H})}<
\frac{\varepsilon}{4},\quad f\in K.
$$
By uniform tightness of
$K$, there exists a finite-dimensional subspace $V\subset\mathcal{H}$ such that
$$
\|f-Q_V\!\circ f\|_{L^2(G,\mathcal{H})}<\frac{\varepsilon}{4},\quad f\in K,
$$
where $Q_V:\mathcal{H}\to V$ is the orthogonal projection on $V$. Since $Q_V$ is a bounded linear map on
$\mathcal{H}$, it commutes with the Bochner integral defining $\widehat f(\sigma)$. Indeed,
\begin{align*}
\widehat{Q_V\circ f}(\sigma)(\xi,\eta)&=\int_G\langle \sigma(x)^* \xi,\eta\rangle_{H_\sigma}Q_V(f(x))dx\\
&=Q_V\left(\int_G\langle \sigma(x)^* \xi,\eta\rangle_{H_\sigma}f(x)dx\right)\\
&=Q_V(\widehat f(\sigma)(\xi,\eta)).
\end{align*}
Hence, $\widehat{Q_V\circ f}(\sigma)=Q_V\big(\widehat f(\sigma)\big)$. Applying the definition of $P_S$, we obtain
$P_S(Q_V\!\circ f)=Q_V\!\circ(P_Sf)$. Since $Q_V$ is an orthogonal projection
(hence norm-nonincreasing), we have 
\begin{align*}
\|Q_V\!\circ f-P_S(Q_V\!\circ f)\|_{L^2(G,\mathcal{H})}&=\|\widehat{Q_V\circ f}\|_{\mathscr{S}_2(\widehat{G}\setminus
S,\mathcal{H})}\\
&\le\|\widehat f\|_{\mathscr{S}_2(\widehat{G}\setminus S,\mathcal{H})}<\frac{\varepsilon}{4},\quad f\in K.
\end{align*}

For $f\in K$, we have 
\begin{align*}
\|f-Q_V(P_Sf)\|_{L^2(G,\mathcal{H})}&\le\|f-Q_V\!\circ f\|_{L^2(G,\mathcal{H})}+\|Q_V\!\circ f-Q_V(P_Sf)\|_{L^2(G,\mathcal{H})}\\
&<\frac{\varepsilon}{4}+\frac{\varepsilon}{4}=\frac{\varepsilon}{2},\quad f\in K.
\end{align*}

\noindent The set $\mathcal{A}:=\{Q_V(P_Sf):f\in K\}$ lies in the linear span of the
finitely many functions $u^\sigma_{ij}\otimes v$ with $\sigma\in S, \, 1\leq i,j\leq d_\sigma$ and $v$ ranging over a basis of the finite-dimensional subspace $V$. Therefore,  $\mathcal{A}$ embeds in a finite-dimensional subspace of $L^2(G,\mathcal{H})$. Since $P_S,Q_V$ are both norm-nonincreasing and $K$ is bounded,
$\mathcal{A}$ is a bounded subset of this finite-dimensional space, hence precompact,  i.e.\  totally bounded. There exist $\{g_1,\dots,g_k\}\subset L^2(G,\mathcal{H})$ such that $\mathcal{A}\subset \union\limits_{i=1}^kB(g_i,\frac{\varepsilon}{2})$. 

 For $f\in K$, choose $g_j$ with $\|Q_V(P_Sf)-g_j\|_{L^2(G,\mathcal{H})}
<\frac{\varepsilon}{2}$.  Then, 
\begin{align*}
\|f-g_j\|_{L^2(G,\mathcal{H})}&\le\|f-Q_V(P_Sf)\|_{L^2(G,\mathcal{H})}+
\|Q_V(P_Sf)-g_j\|_{L^2(G,\mathcal{H})}\\
&<\frac{\varepsilon}{2}+\frac{\varepsilon}{2}=\varepsilon.
\end{align*}
 Hence, $K$ is totally bounded, i.e.\ precompact.
\end{proof}

\section*{Toward Applications}
Pego's original theorem carried, alongside its analytic content, a remark of an information-theoretic flavor: the rate at which a precompact family's Fourier coefficients decay controls the number of degrees of freedom needed to approximate every member of the family to within a prescribed accuracy. The vector-valued theorems established here suggest an analogous reading for group-structured data that is genuinely multi-channel rather than scalar. A bounded, uniformly $L^2(G,\mathcal{H})$-equicontinuous family of $\mathcal{H}$-valued signals on a compact group $G$ for instance, vector or spectrum-valued measurements indexed by an angular or periodic variable is  compressible by finitely many modes whenever $\mathcal{H}$ itself is finite-dimensional. Theorem~\ref{thm:main-tight} extends this to the infinite-dimensional setting once the family is additionally uniformly tight, i.e. once its essential content can be captured, uniformly across the family, by a single finite-dimensional subspace of $\mathcal{H}$.

\subsection*{Degrees of freedom}

The proofs of Theorem~\ref{thm:main} and Theorem~\ref{thm:main-tight} are constructive: given
$\varepsilon>0$, they exhibit an explicit finite index set $S\subset\widehat G$
such that every $f\in K$ is within $\frac{\varepsilon}{2}$
of its band-limited truncation $P_Sf$. This yields an explicit count of how
many scalar coefficients are needed to represent every member of $K$ to a
prescribed accuracy.

\begin{corollary}[finite-dimensional case]
\label{cor:dof-finite}
Let $\mathcal{H}$ be a complex Hilbert space with $\dim \mathcal{H} = d < \infty$, and let
$K\subset L^2(G,\mathcal{H})$ be bounded. For every $\varepsilon>0$ there is a finite
set $S\subset\widehat G$, with
\[
N(S) := \sum_{\sigma\in S} d_\sigma^2,
\]
such that every $f\in K$ is determined, up to an $L^2(G,\mathcal{H})$-error of at most
$\varepsilon$, by its $N(S)\times d$ scalar Fourier coefficients
\[
\big\{\langle \widehat f(\sigma)_{ij}, h_k\rangle_\mathcal{H} \;:\; \sigma\in S,\
1\le i,j\le d_\sigma,\ 1\le k\le d\big\},
\]
where $\{h_k\}_{k=1}^d$ is any fixed orthonormal basis of $\mathcal{H}$.
\end{corollary}

\begin{proof}
This is immediate from the proof of Theorem~\ref{thm:main} because by uniform
$\mathscr{S}_2(\widehat{G},\mathcal{H})$-decay there is $S$ such that
$\|f-P_Sf\|_{L^2(G,\mathcal{H})}<\varepsilon$ for all $f\in K$, and by
Lemma~\ref{lem:PS-finite-dim} the range $P_S\big(L^2(G,\mathcal{H})\big)$ is linearly isometric to
$\mathcal{H}^{N(S)}\cong \mathbb C^{N(S)\times d}$, coordinatized exactly by the coefficients
listed above.
\end{proof}

\begin{corollary}[uniformly tight case]
\label{cor:dof-tight}
Let $\mathcal{H}$ be any complex Hilbert space and let $K\subset L^2(G,\mathcal{H})$ be bounded
and uniformly tight. For every $\varepsilon>0$ there exist a finite set
$S\subset\widehat G$ and a finite-dimensional subspace $V\subset \mathcal{H}$,
$\dim V = m$, such that every $f\in K$ is determined up to an
$L^2(G,\mathcal{H})$-error of at most $\varepsilon$ by $N(S)\times m$ scalar
coefficients.
\end{corollary}

\begin{proof}
Immediate from the proof of Theorem~\ref{thm:main-tight}, combining the band-limiting projection $P_S$
 with the finite-dimensional vector space $V$ arising from the  tightness hypothesis.
\end{proof}

These are the vector-valued analogues of the classical statement that a
translation-equicontinuous, Fourier-decaying family in $L^2(\mathbb{R}^n)$
can be approximated uniformly by $N$-term truncations, with $N$ depending
only on $\varepsilon$ and the family, not on the individual function. The
new content here is that when the values live in a Hilbert space rather
than $\mathbb{C}$, the same finite frequency set $S$ suffices for
every member of $K$, but the per-frequency cost multiplies by the
value-space dimension: $N(S)$ group-theoretic degrees of freedom times $d$
(or $m$, under tightness) value-space degrees of freedom. Concretely, for
$\mathcal{H}$-valued data indexed by a compact group, for instance,
multi-component measurements sampled over $SO(3)$ or a finite group of
symmetries,  Corollary~\ref{cor:dof-finite} says that a bounded,
equicontinuous family can be captured by finitely many complex numbers per
representation, and Corollary~\ref{cor:dof-tight} extends this to
infinite-dimensional value spaces at the cost of the extra tightness
hypothesis.

\subsection*{Algorithmic and numerical aspects}

Corollaries~\ref{cor:dof-finite} and \ref{cor:dof-tight} are existence
statements: for each $\varepsilon>0$ they guarantee a finite $S$ (and,
under tightness, a finite-dimensional $V$) that works uniformly over $K$,
but neither the proof of Theorem~\ref{thm:main} nor of Theorem~\ref{thm:main-tight} gives a
constructive rule for finding $S$ or $V$ from data, nor a quantitative
relationship between $\varepsilon$ and the cardinality  $|S|$ of $S$. What follows separates two
things clearly: (a) what can be computed exactly once $S$ and $V$ are
given, which is a direct consequence of the linear-algebraic content
of Lemma~\ref{lem:PS-finite-dim}; and (b) how one might
choose $S$ and $V$ from a finite sample of $K$ in practice, which is a
genuine estimation problem and is not solved by the theorems above. It requires additional statistical assumptions.

\subsubsection*{(a) Computing the truncation once $S,V$ are fixed}

Suppose $G$ is finite, or more generally that we work with a finite
quadrature rule $\{(x_\ell,w_\ell)\}_{\ell=1}^L\subset G\times\mathbb R_{>0}$
exact for the trigonometric-polynomial degree needed to reproduce the
matrix coefficients $u^\sigma_{ij}$ for $\sigma\in S$.
Given samples $f(x_1),\dots,f(x_L)\in \mathcal{H}$, the truncated Fourier
coefficients are the finite sums
\begin{equation}
\widehat f(\sigma)_{ij} \;\approx\; \sum_{\ell=1}^{L} w_\ell\, u^\sigma_{ij}(x_\ell)\, f(x_\ell),
\quad \sigma\in S,\ 1\le i,j\le d_\sigma,
\end{equation}
each an $\mathcal{H}$-valued linear combination of the $L$ samples. If in addition we
fix an orthonormal basis $\{v_1,\dots,v_m\}$ of $V$, the scalar coefficients
of Corollary~\ref{cor:dof-tight} are
\[
c^\sigma_{i,j,k} = \big\langle \hat f(\sigma)_{ij}, v_k \big\rangle_\mathcal{H},
\qquad 1\le k\le m,
\]
and $P_S(Q_Vf)$ is recovered from the $N(S)\times m$ complex numbers $c^\sigma_{i,j,k} $
by the finite inversion sum of Formula~(\ref{eq:inversionformula}). Computing all coefficients for one $f$ costs
$O(L\times N(S)\times \dim V)$ scalar operations if $V$ is represented in a
fixed basis, and reconstructing $P_S(Q_Vf)$ at a new point costs
$O(N(S)\times \dim V)$. None of this requires new theory; it is bookkeeping
around Lemma~\ref{lem:PS-finite-dim}.

\subsubsection*{(b) Choosing $S$ and $V$ from data.}

The theorems give no rate: uniform $\mathscr{S}_2(\widehat{G},\mathcal{H})$-decay only says that the tail
$\|f\|_{\mathscr{S}_2(\widehat{G}\setminus S,\mathcal{H})}$ can be made small for some finite set $S$, with no control on $|S|$ as a function of $\varepsilon$ beyond what is
implicit in $K$ itself. Any concrete choice of $S$ is therefore a modeling
decision, not a consequence of the present paper's results, unless one is willing
to assume an explicit decay rate (for instance,  $\sup\limits_{f\in K}\|\widehat{f} (\sigma)\|_{B_2(H_\sigma)}
\le C(1+d_\sigma)^{-\alpha}$ for some known $\alpha>0$, which the paper
does not assume or establish). With that caveat stated plainly, a natural
data-driven scheme is the following:

\begin{enumerate}
\item \textbf{Frequency selection.} Given samples $f_1,\dots,f_n\in K$, estimate
$e(\sigma) := \max\limits_{1\le r\le n} \|\widehat{f_r}(\sigma)\|_{B_2(H_\sigma)}$
for the finitely many $\sigma$ with $d_\sigma$ below some cutoff, and take
$S_\varepsilon = \{\sigma : e(\sigma) \ge \delta(\varepsilon)\}$ for a
threshold $\delta$ chosen e.g.\ by cross-validation. This is a heuristic
analogue of the existence argument in Lemma~\ref{lem:decay}/Lemma~\ref{lem:equicont} and inherits no
guarantee from them beyond the qualitative fact that such an $S$ exists
somewhere.

\item \textbf{Value-space selection via an empirical covariance operator.}
Uniform tightness (Definition~\ref{def:tight}) asks for a single finite-dimensional
$V\subset \mathcal{H}$ capturing all of $K$ to within $\varepsilon$ in $L^2(G,\mathcal{H})$-norm. A standard way to estimate such a $V$ from samples
$f_1,\dots,f_n$ is the (Bochner) empirical covariance operator
$$
\widehat{C} := \frac{1}{n}\sum_{r=1}^n \int_G (f_r(x)-\overline{f}(x))\otimes
(f_r(x)-\overline{f}(x))\, dm_G(x) \text{ with }   \overline{f}  = \frac{1}{n}\sum\limits_{r=1}^n f_r.
$$
The operator $\widehat{C}$ is a trace-class, self-adjoint, positive operator on $\mathcal{H}$ whenever the $f_r$
are bounded in $L^2(G,\mathcal{H})$. One may   take $\widehat{V}$ (the estimation of $V$) equals to the  span of the  top $m$
eigenvectors of $\widehat{C}$ (a Karhunen-Lo\`eve/Principal Component Analysis (PCA) truncation)\cite{karhunen1947,loeve1948}. This is the
mean-square-optimal $m$-dimensional subspace for the $n$ sampled
functions, by the standard optimality property of PCA; it is not, without
further argument, the same as the $V$ furnished by Definition~\ref{def:tight} for the
whole (possibly infinite, possibly not identical to $\{f_1,\dots,f_n\}$)
set $K$. Extending the sample estimate to a guarantee for all of $K$ would
require a concentration or covering argument relating $\{f_1,\dots,f_n\}$
to $K$, which we do not supply.

\item \textbf{Joint error accounting.} 
If both steps above are performed, the
triangle inequality gives
$$
\|f - P_{S_\varepsilon}(Q_{\widehat V} f)\|_{L^2(G,\mathcal{H})}
\;\le\; \|f - Q_{\widehat V}f\|_{L^2(G,\mathcal{H})} + \|Q_{\widehat V}f -
P_{S_\varepsilon}(Q_{\widehat V}f)\|_{L^2(G,\mathcal{H})},
$$
mirroring the proof of Theorem~\ref{thm:main-tight}; but turning this into a numerical
bound requires the two heuristic steps above to actually control their
respective terms on the full set $K$, not just on the $n$ samples
used to construct $S_\varepsilon$ and $\widehat{V}$,  a gap between
in-sample and out-of-sample guarantees that is a genuine open point, not a
detail we are glossing over.
\end{enumerate}

\noindent A rigorous treatment
would need either  an explicit, verifiable decay-rate assumption on $K$
replacing the qualitative uniform decay of Definition~\ref{def:decay}, or  a
learning-theoretic argument (e.g.\ a covering-number or Rademacher-complexity
bound for $K$\cite{bartlett2002})  converting the finite-sample estimate  into a
guarantee over $K$. Both are outside the scope of the present paper and we
leave them as open directions.


\begin{thebibliography}{9}

\bibitem{Assiamoua1989} V.S.K. Assiamoua, A. Olubummo, Fourier-Stieltjes transform of vector-valued measures on compact groups, Acta Sci. Math. (Szeged), {\bf 53} (1989), 301-307.

\bibitem{bartlett2002} P.L. Bartlett and S. Mendelson, Rademacher and Gaussian complexities: risk bounds and
structural results, J. Mach. Learn. Res. {\bf 3} (2002), 463-482.
\bibitem{BerghLofstrom}
J. Bergh and J. L\"ofstr\"om, Interpolation Spaces: An Introduction,
Grundlehren der Mathematischen Wissenschaften \textbf{223}, Springer-Verlag, Berlin-New York, 1976.

\bibitem{Dorfler2002} M. D\"orfler, H. G. Feichtinger and K. Gr\"ochenig, Compactness criteria in function spaces,
Colloq. Math., {\bf 94}(1) (2002), 37-50.

\bibitem{Folland} G.B. Folland, A course in abstract harmonic analysis,  CRC Press, Inc, Boca Raton, 1995. 


\bibitem{GarciaRubio}
J. Garc\'ia-Cuerva and J.L. Rubio de Francia, Weighted Norm Inequalities and
Related Topics, North-Holland Mathematics Studies \textbf{116}, North-Holland,
Amsterdam, 1985.

\bibitem{Gorka2014} P. G\'orka, Pego theorem on locally compact abelian groups, J. Algebra Appl. \textbf{13}(4) (2014), art. id. 1350143.


\bibitem{Gorka2016} P. G\'orka and T. Kostrzewa, Pego everywhere, J. Algebra Appl., {\bf 15}(4) (2016), art. id. 1650074.

\bibitem{Gorka2019} P.  G\'orka and P. P\'ospiech, Banach function spaces on locally compact groups, Ann. Funct.
Anal. {\bf 10}(4) (2019), 460-471.



\bibitem{Hanche-Olsen} H. Hanche-Olsen and H. Holden, The Kolmogorov-Riesz compactness theorem, Expo. Math.
{\bf 28}(4) (2010), 385-394.

\bibitem{Horvath2022} A.P. Horv\'ath, Compactness criteria via Laguerre and Hankel transformations, J. Math. Anal.
Appl.,  {\bf 507}(2) (2022), art. id. 125852.

\bibitem{karhunen1947}K. Karhunen, \"Uber lineare Methoden in der Wahrscheinlichkeitsrechnung, Ann. Acad. Sci. Fenn. Ser. A. I. Math.-Phys. {\bf 37}  (1947), 3-79.



\bibitem{Krukowski2020} M. Krukowski, Characterizing compact families via the Laplace transform, Ann. Acad. Sci.
Fenn. Math., {\bf 45}(2) (2020), 991-1002.

\bibitem{Krukowski} M. Krukowski, How Arzel\`a and Ascoli would have proved Pego theorem for $L^1(G)$ (if they
lived in the 21st century)?, preprint, 2020. arXiv 2006.12130.




\bibitem{Kumar2024} M. Kumar, \emph{Pego theorem on compact groups}, Pacific J. Math.
\textbf{328}(1) (2024), 137-143.

\bibitem{loeve1948} M. Lo\`eve, Fonctions al\'eatoires du second ordre, in P. L\'evy, Processus Stochastiques et Mouvement
Brownien, Gauthier-Villars, Paris, 1948, 366-420.

\bibitem{Mensah2024} Y. Mensah, Vector Fourier analysis on compact groups and Assiamoua spaces,  Methods Funct. Anal. Topol., {\bf 30}(3-4) (2024), 147-154.


\bibitem{Peetre} J. Peetre, Sur la transformation de Fourier des fonctions \`a valeurs vectorielles, Rend. Sem. Mat. Univ. Padova \textbf{42} (1969), 15-26.

\bibitem{Pego} R.L. Pego, Compactness in $L^2$ and the Fourier transform, Proc.
Amer. Math. Soc. \textbf{95} (1985), 252-254.

\bibitem{Weil} A. Weil, L'int\'egration dans les Groupes Topologiques et ses Applications, Hermann, Paris, 1940.


\end{thebibliography}
\end{document}